\documentclass[11pt]{article}
\usepackage{fullpage}
\usepackage{subfig}
\usepackage{times}
\usepackage{graphicx}
\usepackage{psfrag}
\usepackage{amsmath,amsthm}
\usepackage{amsfonts}
\newtheorem{thm}{Theorem}[section]

\newtheorem{cor}[thm]{Corollary}
\newtheorem{lem}[thm]{Lemma}

\theoremstyle{definition}
\newtheorem{defn}{Definition}[section]

\theoremstyle{remark}

\begin{document}

\parskip=10pt

\flushbottom 

\title{Upper bound on the packing density of regular tetrahedra and octahedra} 

\author{Simon Gravel\\
Department of Genetics\\ Stanford University School of Medicine\\ Stanford, California, 94305-5120
\and Veit Elser and Yoav Kallus\\
Laboratory of Atomic and Solid State Physics\\
Cornell University\\ Ithaca, NY 14853-2501} 

\date{}

\maketitle

\begin{abstract}
We obtain an upper bound to the packing density of regular tetrahedra. The bound is obtained by showing the existence, in any packing of regular tetrahedra, of a set of disjoint spheres centered on tetrahedron edges, so that each sphere is not fully covered by the packing. The bound on the amount of space that is not covered in each sphere is obtained in a recursive way by building on the observation that non-overlapping regular tetrahedra cannot subtend a solid angle of $4\pi$ around a point if this point lies on a tetrahedron edge. The proof can be readily modified to apply to other polyhedra with the same property. The resulting lower bound on the fraction of empty space in a packing of regular tetrahedra is $2.6\ldots\times 10^{-25}$ and reaches $1.4\ldots\times 10^{-12}$ for regular octahedra.

  \end{abstract}

\section{Introduction}

The problem of finding dense arrangements of non-overlapping objects, also known as the packing problem, holds a long and eventful history and holds fundamental interest in mathematics, physics, and computer science. Some instances of the packing problem rank among the longest-standing open problems in mathematics.

The archetypal difficult packing problem is to find the arrangements of non-overlapping, identical balls that fill up the greatest volume fraction of space. The face-centered cubic lattice was conjectured to realize the highest packing fraction by Kepler, in 1611, but it was not until 1998 that this conjecture was established using a computer-assisted proof \cite{Hales:2005p874} (as of March 2009, work was still in progress to ``provide a greater level of certification of the correctness of the computer code and other details of the proof" \cite{Hales:2010p949}) .

Another historically important problem is the densest packing of the five platonic (regular) solids:  the tetrahedron, cube, octahedron, dodecahedron, and icosahedron (in the following, these terms refer to the regular solids only). The proof that exactly five regular solids exist was an important achievement of ancient Greek geometry \cite{Euclid:1908p976}, and its perceived significance at the time is reflected in Plato's theory of matter, which used them as fundamental building blocks: the four elements of air, earth, fire, and water were taken to be composed of particles with octahedral, cubic, tetrahedral, and icosahedral shapes, respectively. The dodecahedron was associated with the cosmos \cite{Senechal:1981p312}.

The problem of packing the platonic solids came with Aristotle's dismissal of the platonic theory of matter. In Aristotle's view, the elementary particles cannot leave space unoccupied and, therefore:

\begin{quotation}

In general, the attempt to give a shape to each of the simple bodies is unsound, for the reason, first, that they will not succeed in filling the whole. It is agreed that there are only three plane figures which can fill a space, the triangle, the square, and the hexagon, and only two solids, the pyramid and the cube. But the theory needs more than these because the elements which it recognizes are more in number. \cite{DeCaelo} 
\end{quotation}

In fact, the cube is the only space-filling "simple body"; the observation that the tetrahedron does not fill space came in the $15^\textrm{th}$ century \cite{Senechal:1981p312}. The sum of the solid angles subtended by tetrahedra around a point cannot add to $4\pi$ if this point is located on a tetrahedron edge or vertex. This guarantees a nonzero amount of empty space in the vicinity of each vertex and along each edge, but does not by itself yield a non-trivial upper bound for the packing density

Hilbert included both the optimal sphere packing and tetrahedron packing problems as part of his $18^\textrm{th}$ problem in 1900 \cite{Hilbert:1900p1142}. Whereas the three dimensional sphere packing problem was resolved by Hales' 1998 proof, the tetrahedron problem remains wholly unresolved. Contrary to the case of the sphere, where the optimal packing structure has been known for centuries, improved tetrahedron packing arrangements keep being uncovered by numerical searches  \cite{Conway:2006p571,HajiAkbari:2009p610, Torquato:2009p747,Torquato:2009p714,Kallus:2010p497,kallus2010method, Chen:2008p521,PhysRevE.81.041310,Chen:2010p783}. The rotational degrees of freedom, irrelevant in the sphere packing case, complicate both numerical and analytical investigation. Also contrary to the case of spheres, the optimal tetrahedron packing density cannot be obtained by a (Bravais) lattice packing. The optimal packing density for a tetrahedral lattice packing is $18/49=0.367\ldots$\cite{Hoylman:1970p1203}, far below the current densest known packing with density $0.856347\ldots$\cite{Chen:2010p783}. The latter packing is periodic, but with four tetrahedra in the fundamental cell. In addition to theoretical and numerical investigations, rigid tetrahedron packing has also received recent experimental attention \cite{Jaoshvili:2010p1551}. The packings found in this study suggest, via extrapolation to large container sizes, random packing densities of $0.76\pm0.02$.  

Given the complexity of the proof of Kepler's conjecture and the additional challenges presented by the tetrahedron problem, obtaining a tight upper bound $\phi$ to the tetrahedron packing density appears a formidable task. In this context, a reasonable starting point would be to bound the optimal density away from $1$, that is, to find a nontrivial upper bound $\hat \phi$ such that $\phi\leq \hat \phi<1$. 

 Interestingly, although the solid angle argument entails that such a nontrivial bound exists, it does not provide a value for $\hat \phi$. Even though a valid argument for the existence of a bound was proposed more than 500 years ago, and the problem of establishing a non-trivial upper bound to the packing density for the regular polyhedra has received increased attention recently \cite{PhysRevE.81.041310,Torquato:2009p747}, we are not aware that an explicit value for such a bound has ever been reported. The simple upper bound strategy proposed by Torquato and Jiao \cite{Torquato:2009p747}, which applies Kepler's conjecture to  spheres inscribed in the polyhedra forming the packing, is successful in providing meaningful bounds for many polyhedra, but fails to provide a nontrivial bound for polyhedra, such as the tetrahedron or the octahedron, whose inscribed sphere occupies too small a fraction of the polyhedron volume.  

In this article, we obtain an explicit bound to the packing density of regular tetrahedra, namely $\phi\leq \hat\phi= 1-\delta$, with $\delta=2.6\ldots\times 10^{-25}.$ In Section \ref{oct}, we explain how the proof can be modified to apply to regular octahedra, and find an upper bound of $ \hat\phi=1-\delta_o$, with $\delta_o=1.4\ldots\times 10^{-12},$ to the packing fraction of regular octahedra. These bounds are certainly not tight, as we have chosen simplicity of the proof at the expense of tightness in the bound. In fact, we conjecture that the optimal packing density corresponds to a value of $\delta$ many orders of magnitude larger than the one presented here. We propose as a challenge the task of finding an upper bound with a significantly larger value of $\delta$ (e.g., $\delta>0.01$) and the development of practical computational methods for establishing informative upper bounds.

\section{Structure of the upper bound argument and definitions}

In order to obtain a bound to the packing density, we show the existence, in any tetrahedron packing, of a set of disjoint balls whose intersection with the packing is particularly simple, and whose density can be bounded below. The construction is such that the density of the packing within each of the balls can be bounded away from one. The combination of these two bounds gives the main result.

More precisely, we show in Theorem \ref{embroidery} the existence of a set of non overlapping balls centered around tetrahedron edges, such that each ball is free of vertices and overlaps with at most $5$ tetrahedra. In Theorem  $\ref{capwedgebound},$ we obtain a lower bound to the unoccupied volume contained in any of the balls by building on the solid angle argument: each possible arrangement of tetrahedra in a ball $B$ is compared to a finite set of scale invariant arrangements, that is, tetrahedron arrangements whose intersection with the ball is invariant by dilation of the tetrahedra about the sphere center. The comparison of the intersection $K^0$ of the packing with $B$ to one of these scale-invariant arrangements can yield two results: if $K^0$ is "close" (in a sense to be specified below) to one of the scale-invariant arrangements $K'$, the unoccupied volume $\operatorname{vol}(B\setminus K^0)$ is close to the unoccupied volume $\operatorname{vol}(B \setminus K')$, and it can be bounded below. Otherwise, a smaller ball $B'$ exists, whose intersection with the packing is a configuration $K^1$ simpler than $K^0.$ The unoccupied volume $\operatorname{vol}(B'\setminus K^1)$ provides a lower bound to the unoccupied volume $\operatorname{vol}(B\setminus K^0)$.  By iterating this procedure, we construct a finite sequence of concentric balls and configurations $\left\{K^i\right\}$ reminiscent of matryoshkas (nested dolls). After a finite number of steps, we are left with a configuration $K^n$ whose unoccupied volume can be bounded analytically. The bound for all possible configurations is then constructed from this bound in a way reminiscent of dynamic programming.

\begin{defn} \label{capswedges}
An infinite wedge is the intersection of two half-spaces. The edge of this infinite wedge is the intersection of the boundaries of the half-spaces. A $B$-wedge is the intersection of an infinite wedge and a ball $B$. The edge of a $B$-wedge is the intersection of the edge of the infinite wedge and $B$. A $B$-wedge is centered if it has its edge along a diameter of $B$ (in which case it is a standard spherical wedge). Unless otherwise stated, we will consider only wedges with the tetrahedron dihedral angle $\arccos(1/3),$ and define $\alpha= \arccos(1/3)/2\pi,$ the fractional solid angle subtended by a tetrahedron edge. A $B$-cap is the intersection of a half-space and a ball $B$.  A centered $B$-cap is a hemisphere in $B$.
\end{defn}

According to Definition \ref{capswedges}, a $B$-cap is a special case of a $B$-wedge. Both the empty set and the ball $B$ are special cases of a $B$-cap. Moreover, the edge of a $B$-wedge can be empty. A centered $B$-wedge occupies an equal fraction $\alpha$ of the volume and of the surface area of $B.$

\begin{defn} \label{defd} Let $\mathcal{K}(c,w,r)$ be the set of all packing (i.e., non-overlapping) configurations of at most $c$ $B$-caps plus at most $w$ $B$-wedges in a unit ball $B$, where at least one $B$-wedge is centered and all $B$-wedge edges and $B$-caps are at a distance of at most $r$ from the center of $B$. Define the minimum missing volume fraction for such arrangements as
\begin{equation}
\delta_r(c,w)\equiv 1-\sup_{K\in \mathcal{K}(c,w,r)}\frac{\mathrm{vol}(K\cap B)}{\mathrm{vol}(B)}.
\end{equation}
The distance of an empty set to any point is $+\infty$.  Distances to the origin of subsets of $B$ can therefore take values in $[0,1] \cup \{\infty\}$. The set $\mathcal{K}(c,w,\infty)$ contains all packing configurations of at most $c$ $B$-caps plus at most $w$ $B$-wedges (including a centered one) in a unit ball $B$, without a constraint on the distance to center. Finally, for later convenience, we define $\mathcal{K}(c,w,r)=\emptyset $ and $\delta_r(c,w)=1$ when $c<0$ or $w<1$. \end{defn}

\begin{figure}
\begin{center}
\includegraphics[clip=true,viewport=.5in 1in 5in 5.5in ]{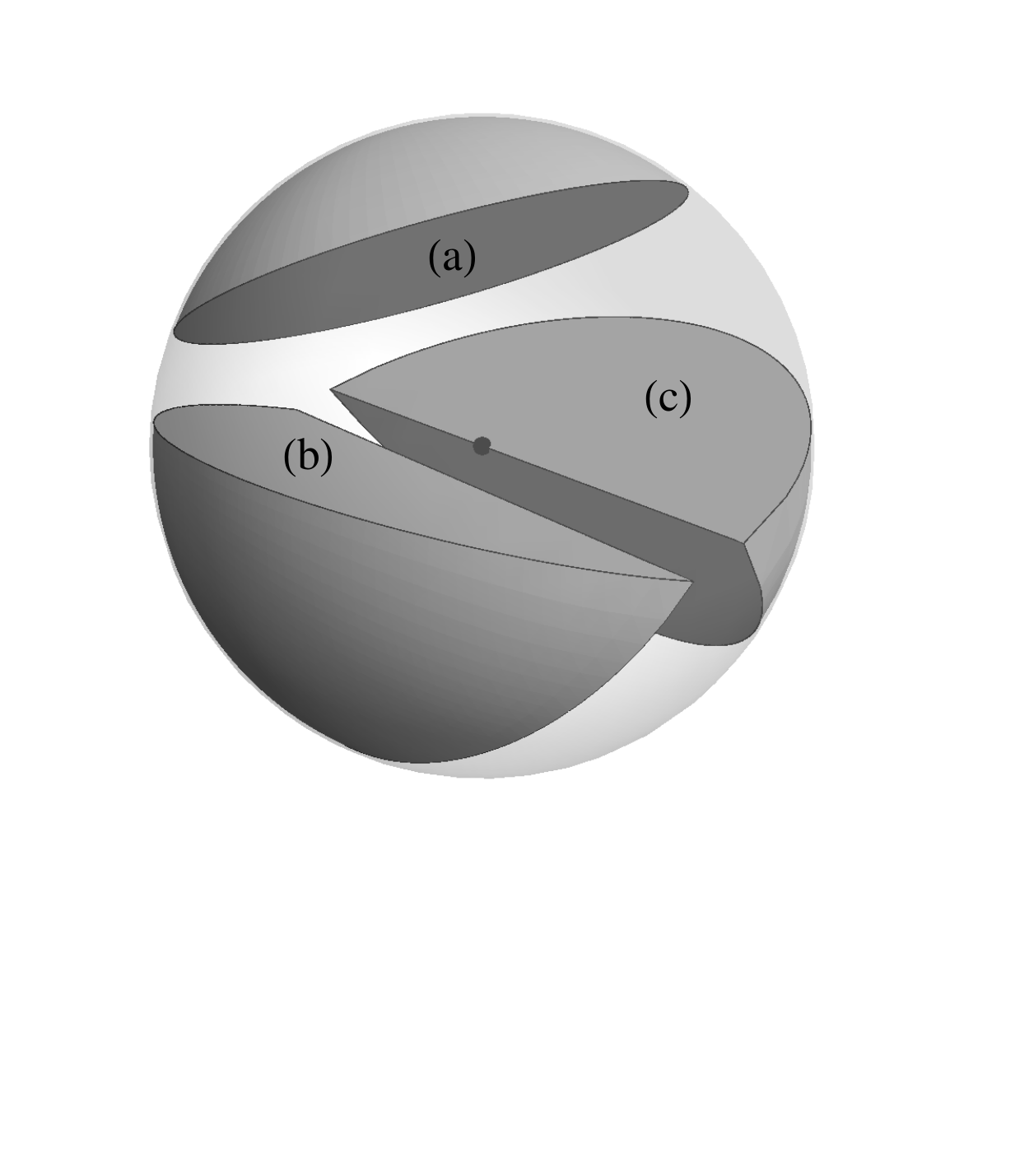}
\end{center}
\caption{\label{defs}  A unit ball $B$ showing a configuration $K\in\mathcal{K}(1,2,0.6)$ consisting of (a) a $B$-cap, (b) a $B$-wedge, and (c) a centered $B$-wedge}

\end{figure}

\begin{defn} \label{order}
Given two pairs of nonnegative integers $(c,w)$ and $(c',w')$, we define the partial order $(c,w)\leq (c',w')$ if $w\leq w'$ and  $c+w \leq c'+w'$.
\end{defn}


\begin{lem} \label{capspecial} 
If $(c,w)\leq(c',w')$,  then 
$\mathcal{K}(c,w,\infty)\subseteq \mathcal{K}(c',w',\infty),$ and $\delta_\infty(c,w)\geq\delta_\infty(c',w')$. 
Furthermore, if $(c,w)<(c',w')$, then at least one of the following is true: $(c,w)\leq(c'-1,w')$, or $(c,w)\leq(c'+1,w'-1)$.\end{lem}

\begin{proof}

Let $(c,w)\leq(c',w')$. If $w<1$, then $\mathcal{K}(c,w,\infty)=\emptyset \subseteq \mathcal{K}(c',w',\infty).$ Otherwise, since the empty set is a special case of a $B$-cap, and a $B$-cap is a special case of a $B$-wedge, $K \in \mathcal{K}(c,w,\infty)$ can be  expressed as the subset of $K'\in\mathcal{K}(c',w',\infty)$, where $w'-w$ $B$-wedges are constrained to be $B$-caps, and $c'+w'-c-w$ $B$-caps are constrained to be empty. Therefore $\mathcal{K}(c,w,\infty)\subseteq \mathcal{K}(c',w',\infty)$, and the inequality
$\delta_\infty(c,w)\geq\delta_\infty(c',w')$ follows from this and the definition of $\delta_\infty(c,w)$.
 Finally, if $(c,w)< (c',w')$, then either $w< w'$ and $w+c \leq w'+c'$, in which case $(c,w)\leq (c'+1,w'-1),$ or $w\leq w'$ and $w+c<w'+c',$ in which case  $(c,w)\leq (c'-1,w').$ \end{proof}

\section{Geometric lemmas}

In this section we present simple geometric lemmas describing $B$-wedges and $B$-caps as a function of their proximity to the center of the ball $B$. We also obtain a condition on the distance between vertices of a tetrahedron $T$ and the center of a ball $B$ ensuring that $T\cap B$ is a $B$-wedge. The proofs are elementary; details  for Lemma \ref{wedgemax}, \ref{wedgesurf}, and \ref{2cw} are provided in the Appendix.

\begin{lem}
\label{wedgemax}
Let $B$ be a unit ball and $W$ a $B$-wedge whose interior does not contain the center of $B$.  If the edge of $W$ is at a distance no greater than $r$ from the center of $B$, then
$$\frac{\operatorname{vol}(W)}{\operatorname{vol}(B)}\le \alpha+\frac{3 r}{8} \sin(2 \pi \alpha) +\frac{3 r^2}{8 \pi} \sin(4 \pi \alpha)=\alpha+\sqrt{2}\left(\frac{r}{4}+\frac{r^2}{6 \pi}\right).$$
Similarly, a $B$-cap $C$ whose interior does not contain the center of $B$ satisfies
$$\frac{\operatorname{vol}(C)}{\operatorname{vol}(B)}\le\frac{1}{2}.$$

\end{lem}

\begin{cor} \label{sm1}
\begin{equation} \label{maxbound}\begin{split}
\delta_r(c,w)&\ge1-\left(\alpha w + \frac{c}{2}\right)-\left(\frac{3 r}{8} \sin(2 \pi \alpha) +\frac{3 r^2}{8 \pi} \sin(4 \pi \alpha)\right)\left(w-1\right)\\&=
1-\left(\alpha w + \frac{c}{2}\right)-\sqrt{2} \left(\frac{r}{4}+\frac{r^2}{6 \pi}\right)(w-1).
\end{split}
\end{equation}
This provides a nontrivial bound for small $r$ when $\alpha w + c/2< 1.$

\end{cor}  
\begin{proof}
Since the configurations in $\mathcal{K}(c,w,r)$ contain a centered $B$-wedge, none of the $B$-wedge or $B$-cap interiors can contain the center of $B$, and Lemma \ref{wedgemax} applies (independently) to each of the $(w-1)$ remaining $B$-wedges and the $c$ $B$-caps. Therefore the occupied space in any configuration in $\mathcal{K}(c,w,r)$ is bounded above by $\alpha+(w-1) \left(\alpha+\sqrt{2}\left(\frac{r}{4}+\frac{r^2}{6 \pi}\right)\right)+c/2$, leading to the desired bound on $\delta_r(c,w).$
\end{proof}






\begin{lem}\label{wedgesurf}
Let $B$ be a unit ball and $W$ a $B$-wedge that intersects a ball $B_r$ of radius $r\leq1$ concentric with $B$. 
Then the area  $\sigma(W\cap\partial B)$ of the intersection of $W$ with the surface $\partial B$ of $B$ is bounded by
\begin{equation}\label{lowerboundonwedgesurface}
\frac{\sigma(W\cap\partial B)}{\sigma(\partial B)}\ge \alpha -\frac{r \sin(\pi \alpha)}{2} + \frac{r^2}{4\pi} \sin(2 \pi \alpha)=\alpha-\frac{r}{2 \sqrt{3} } + \frac{\sqrt{2} r^2}{6 \pi}.\end{equation}

Similarly, given a $B$-cap $C$ at a distance less than $r$ from the center of $B$, the surface area  $\sigma(C\cap\partial B)$ is bounded by
$$\frac{\sigma(C\cap\partial B)}{\sigma(\partial B)} \geq \frac{1}{2}\left(1-r\right).$$

\end{lem}

\begin{figure}
\centering
\subfloat[wedge]{\scalebox{.7}{\includegraphics{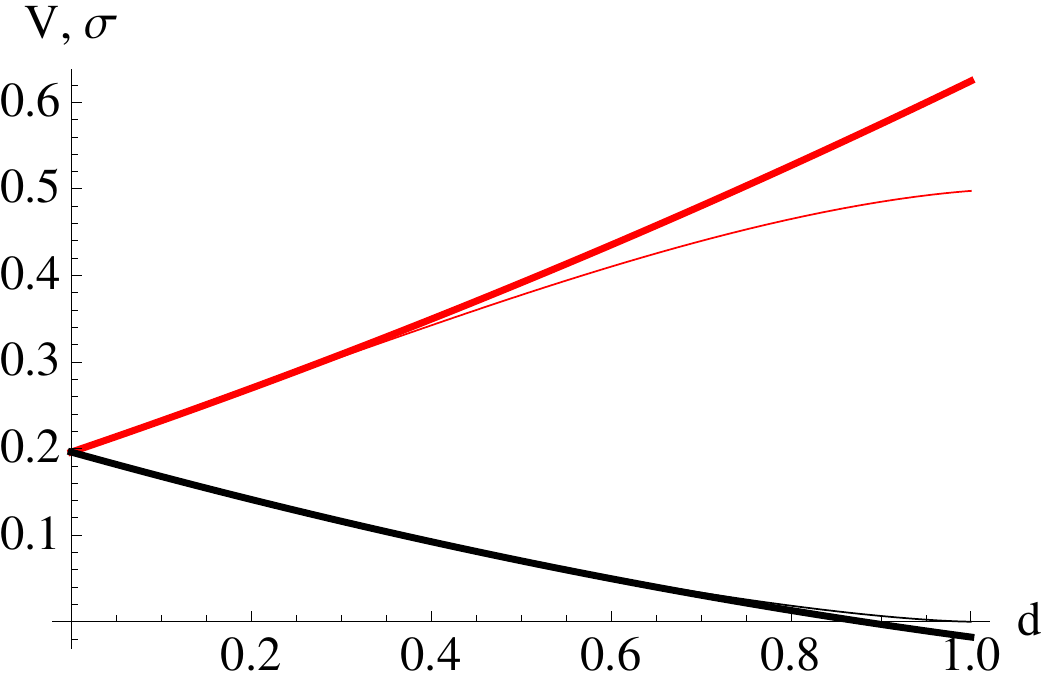}}}
\subfloat[cap]{\scalebox{.7}{\includegraphics{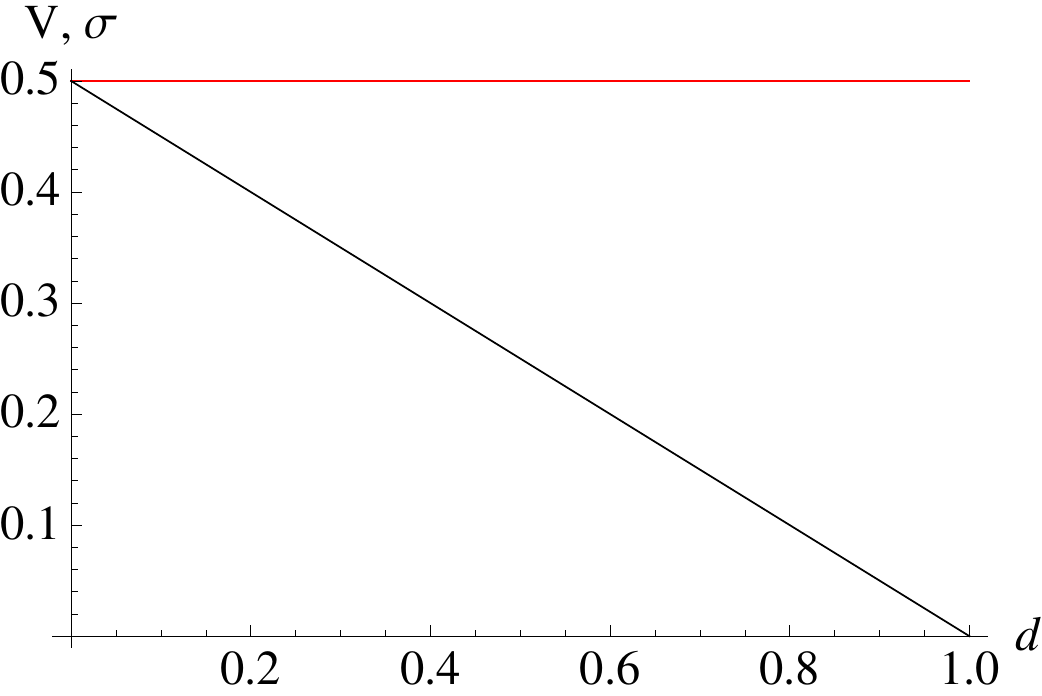}}}\\
\subfloat[cap + 2 wedges]{\scalebox{.7}{\includegraphics{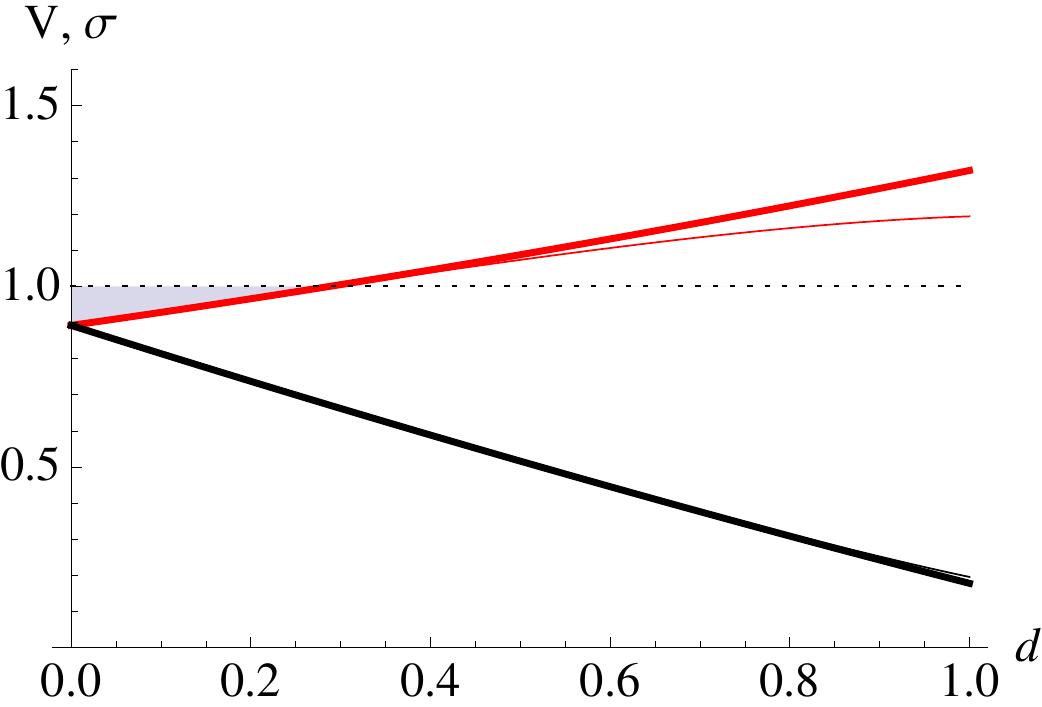}}}
\subfloat[cap + 3 wedges]{\scalebox{.7}{\includegraphics{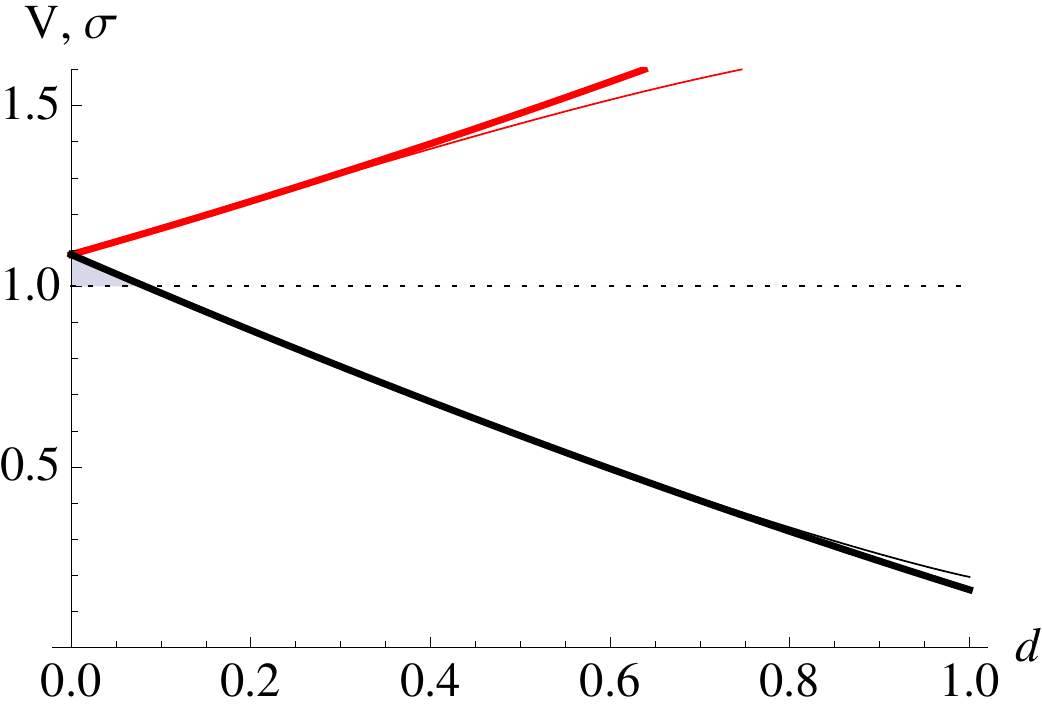}}}
\caption{\label{penguin} Upper bounds on the fractional volume $V=\frac{\operatorname{vol} K}{\operatorname{vol}{B}}$ (red) and lower bounds to the fractional surface $\sigma=\frac{\sigma(K\cap \partial B)}{\sigma(B)}$ (black)  of configurations $K$ of $c$ $B$-caps and $w$ $B$-wedges in the unit ball $B$, where all $B$-caps and the edges of all $B$-wedges are within distance $d$ to the center of $B$, and such that the interior of no $B$-wedge or $B$-cap contain the center of $B$. Thick lines correspond to the bounds from Corollary \ref{sm1} and Lemma \ref{wedgesurf}. Thin lines correspond to improved bounds that could be obtained by avoiding the simplifying approximations used in both results. Thin and thick lines overlap for $B$-caps. The configuration $K$ comprises (a) a single $B$-wedge, (b) a single $B$-cap, (c) one centered $B$-wedge, plus one $B$-wedge and one $B$-cap, (d)  one centered $B$-wedge, plus two $B$-wedges and one $B$-cap. Since $\alpha w+c/2\neq1$ for the tetrahedron, we can find a radius $\hat d (c,w)$ so that $d<\hat d(c,w)$ guarantees a finite amount of missing volume (shaded area, c) or overlap  on $\partial B$ (shaded area, d). }
\end{figure}

\begin{cor}\label{5wedges}
Consider a ball $B$ containing non-overlapping $B$-wedges and $B$-caps, at least one of which is centered, and a concentric ball $B_{\gamma_5}$ whose radius is $\gamma_5\equiv 0.125$ times the radius of $B$. At most five of the non-overlapping $B$-wedges and $B$-caps can intersect  $B_{\gamma_5}$. 

Similarly, at most seven non-overlapping $B$-wedges and $B$-caps can intersect a concentric ball $B_{\gamma_7},$ of radius $\gamma_7\equiv 0.304$ times the radius of $B$.


\end{cor}

\begin{proof}

We prove below the result for a unit ball $B$ (in which case $B_{\gamma_q}$ has radius $\gamma_q$). The result for arbitrary radius follows since a dilation preserves intersections, overlaps, and angles: if a configuration violated the Corollary in a non-unit ball, it could be mapped onto a configuration violating the Corollary for a unit ball, a contradiction. 

The $\gamma_q$ are chosen so that if $q+1$ non-overlapping $B$-wedges intersect $B_{\gamma_q}$, with at least one of them centered, then Lemma \ref{wedgesurf} implies that the total covered fraction $f$ of the surface of $B$ is greater than 1, that is,
$$f\geq (q+1)\alpha+q\left(-\frac{\gamma_q}{2 \sqrt{3} } + \frac{\sqrt{2} \gamma_q^2}{6 \pi} \right)>1,$$ a contradiction since the wedges are non-overlapping.

Therefore, at most $q$ non-overlapping $B$-wedges can have their edges at a distance smaller than $\gamma_q$ from the origin. Since a $B$-cap is a special case of a $B$-wedge, the result also applies to the total number of $B$-wedges and $B$-caps.

\end{proof}

\begin{lem} \label{2cw}

If a unit edge tetrahedron $T$ has all its vertices at a distance greater than $\eta r= 3 r/\sqrt{2} $ from the center of a ball $B_r$ of radius $r$, and if the interior of $T$ does not contain the center of $B_r$, then $T\cap B_r$ is a $B_r$-wedge. 
\end{lem}

\section{Recursion lemmas}
\label{recurs}
The first two lemmas presented in this section bound the missing density of a configuration $K$ in a ball $B$ in terms of the missing density of simpler configurations $K'$ in a ball $B'\subset B$ concentric with $B$. Lemma \ref{overfull} does this for configurations $K\in \mathcal{K}(c,w)$ with  $t\equiv \alpha w + c/2> 1,$ and Lemma \ref{lemmin} for configurations with $t<1$. The case $t=1$ does not occur for the tetrahedron, as can be verified by direct enumeration of pairs $(c,w)$ with $c\leq 1$ and $1\leq w \leq 5$. If the argument below is applied to a polyhedra whose dihedral angle divides $2 \pi$, such as the cube, it results in a trivial bound for the missing density.  Finally, Lemma \ref{simple} establishes a bound on the missing density for the simplest configurations, with $c+w=2$. These three lemmas are the building blocks of the recursion used to find a finite lower bound to $\delta_\infty(0,5)$.

\begin{lem}
\label{recurs1}
If $\alpha w + c/2> 1$,  then

\label{overfull}

\begin{equation}\label{ladderfull}
\delta_\infty(c,w)\ge \kappa(c,w)^3\times \left\{\begin{array}{cc} \delta_\infty(c-1,w) & \operatorname{if} c\neq 0 \\
\delta_\infty(0,w-1) & \operatorname{otherwise},
\end{array}
\right.
\end{equation}
with 
$$\kappa(c,w)\equiv \frac{2 (c/2+ w \alpha -1)}{ c+ (w-1) \sin(\pi \alpha)}\in\left (0,1\right).$$


\end{lem}

\begin{proof}
If $t \equiv \alpha w + c/2 >1$, the surface area of the intersection of $c$ centered $B$-caps and $w$ centered $B$-wedges with $\partial B$ exceeds the surface area of $\partial B$. The minimal fraction of $\partial B$ covered by a $B$-wedge or by a $B$-cap intersecting a ball $B_r$ of radius $r\leq1$ concentric with $B$ was bounded in Lemma \ref{wedgesurf}. If all $B$-caps and $B$-wedges intersect $B_r$ and one wedge is centered, the total fraction $\sigma$ of the surface of $\partial B$ that is covered is at least 

\begin{equation}
\begin{split}
\sigma \geq w \alpha - (w-1)\left( \frac{r \sin(\pi \alpha)}{2}-\frac{\sin(2 \pi \alpha)}{4\pi }r^2\right) + \frac{c}{2}\left(1-r\right)\\
\geq w \alpha - (w-1)\left( \frac{r  \sin(\pi \alpha)}{2}\right) +\frac{c}{2}\left(1- r\right).
\end{split}
\end{equation}
If $r\leq\kappa(c,w)$, we have $\sigma \geq1$. Therefore, if $r= \kappa(c,w),$ at least one $B$-cap or $B$-wedge does not overlap $B_{r}$. We bound the missing volume in $B$ by the missing volume in the ball $B_{ \kappa(c,w)}$ of radius $\kappa(c,w)$ concentric with $B$. The fractional missing volume is unchanged by dilation around the origin of $B_{ \kappa(c,w)},$  and we obtain
\begin{equation}
\delta_\infty(c,w)\ge \kappa(c,w)^3\times \min_{\left\{(c',w')\large|\substack{c'+w'<c+w\\ w'\leq w}\right\}} \left(\delta_\infty(c',w')\right).
\end{equation}
The minimal $\delta_\infty(c',w')$ is obtained by removing one cap ($c'=c-1$) if $c\geq 1$ and one wedge ($w'=w-1$) otherwise. 

\end{proof}

\begin{lem}\label{lemmin} For $0< r\leq1,$
\begin{equation}
\delta_\infty(c,w)\ge \min(\delta_r(c,w),r^3 \hat \rho),
\end{equation}
for any $\hat \rho\leq\rho(c,w)\equiv \min\left(\delta_\infty(c-1,w),\delta_\infty(c+1,w-1)\right).$

In particular, if $\alpha w+c/2<1$ and $\hat \rho>0$, then
 \begin{equation}\label{ladderempty}
 \delta_\infty(c,w)\ge \tau^3(c,w,\hat \rho) \hat \rho,
\end{equation}
where  $\tau(c,w,\hat \rho)$ is the minimum among $1$ and the unique nonnegative value of $r$ satisfying 

\begin{equation}
\label{unique}
\begin{split}
r^3 \hat \rho &=1-(\alpha w + \frac{c}{2})-\left(\frac{3 r}{8} \sin(2 \pi \alpha) +\frac{3 r^2}{8 \pi} \sin(4 \pi \alpha)\right)\left(w-1\right)\\&=1-\frac{c}{2}-\alpha w-\sqrt{2} \left(\frac{r}{4}+\frac{r^2}{6 \pi}\right)(w-1).\end{split}
\end{equation}

\end{lem}

\begin{proof}
Consider a configuration $K\in \mathcal{K} (c,w,\infty)$ of $B$-caps and $B$-wedges, and let $\delta_K$ be its missing volume density. Let $B_r$ be the ball of radius $r$ concentric with $B$. If the $c$ caps and the edges of the $w$ wedges intersect $B_r$,
\begin{equation}\label{close}
\delta_K\geq\delta_r(c,w).
\end{equation}
Otherwise, at least one cap does not overlap $B_r,$ or the edge of one wedge does not overlap $B_r.$ We construct a non-overlapping configuration $K'$ of caps and wedges in the ball $B_r$ in the following way: for each $B$-cap $C$ in $K$ overlapping $B_r$, let $K'$ contain $B_r\cap C$. For each $B$-wedge $W$ of $K$ whose edge intersects the interior of $B_r$, let $K'$ contain $B_r \cap W$. Finally, if a $B$-wedge $W$ overlaps $B_r$, but its edge does not, then either $W\cap B_r$ is a $B_r$ cap, or $B_r\setminus W$ has two disconnected components. In the former case, let $K'$ contain the $B_r$-cap $W\cap B_r$. In the latter case, let $K'$ contain the cap $C'$ given by the union of $W\cap B_r$ with the disconnected component not containing the center of $B_r$. 
The configuration $K'$ is composed of $B_r$-caps and $B_r$-wedges, but the addition of the disconnected components might have resulted in overlaps. As a final step, we remove from $K'$  any $B_r$-cap or $B_r$-wedge contained within these added disconnected components. We have, by construction, $\operatorname{vol}(K'\cap B_r)\geq\operatorname{vol}(K\cap B_r)$. In this procedure, at least one wedge in $K$ was turned into a cap in $K'$ (possibly empty), or a cap in $K$ is not present in $K'$.

Now consider the configuration $K'/r$ in the unit ball $B$ obtained by expanding $K'$ around the center of $B$ by a factor $1/r$. Using Lemma \ref{capspecial}, we have
$$K'/r\in\mathcal{K}(c-1,w,\infty)\cup\mathcal{K}(c+1,w-1,\infty).$$
Therefore, we have 
\begin{eqnarray*}
\delta_K & = & 1-\frac{\mathrm{vol}(K\cap B)}{\mathrm{vol}(B)}\\
& \ge & \frac{\mathrm{vol}(B_r)-\mathrm{vol}(K\cap B_r)}{\mathrm{vol}(B)}\\
& \ge & \frac{\mathrm{vol}(B_r)-\mathrm{vol}(K'\cap B_r)}{\mathrm{vol}(B)}\\
& = & r^3\left(1-\frac{\mathrm{vol}(K'\cap B_r)}{\mathrm{vol}(B_r)}\right)\\
& = & r^3\left(1-\frac{\mathrm{vol}((K'/r)\cap B)}{\mathrm{vol}(B)}\right)\\
& \ge & r^3 \min(\delta_\infty(c-1,w),\delta_\infty(c+1,w-1))=r^3 \rho(c,w).
\end{eqnarray*}
Combining this result with equation \eqref{close} we have, for all $K\in \mathcal{K} (c,w,\infty)$,

\begin{equation}
\delta_K  \ge  \min(\delta_r(c,w),r^3\rho(c,w)).
\end{equation}
Therefore, by definition of $\delta_\infty(c,w) $,
 \begin{equation}
\delta_\infty(c,w)  \ge  \min(\delta_r(c,w),r^3 \rho(c,w))\geq \min(\delta_r(c,w),r^3 \hat \rho) .
\end{equation}
This result is valid for any $0< r\leq 1.$ We are interested in values of $r$ that provide as strong a bound as possible. By Corollary \ref{sm1}, $\delta_r(c,w)\ge \psi(r)=1-(\alpha w + c/2)-\left(\frac{3 r}{8} \sin(2 \pi \alpha) +\frac{3 r^2}{8 \pi} \sin(4 \pi \alpha)\right)\left(w-1\right)=1-\alpha w-c/2-\sqrt{2} \left(\frac{r}{4}+\frac{r^2}{6 \pi}\right)(w-1).$ This is a non-increasing function of $r$ and, when $\alpha w+c/2<1$, satisfies $\psi(0)>0$. In this case, for any $\hat \rho>0$, $r^3 \hat \rho$ is a strictly increasing function of $r$ and there is a unique positive solution $r=s(c,w,\hat \rho)$ to

\begin{equation}
\label{taucond}
r^3 \hat \rho =1-\left(\alpha w + \frac{c}{2}\right)-\left(\frac{3 r}{8} \sin(2 \pi \alpha) +\frac{3 r^2}{8 \pi} \sin(4 \pi \alpha)\right)\left(w-1\right).
\end{equation}

We define $\tau(c,w,\hat \rho)\equiv \min(1,s(c,w,\hat \rho))$  and choose $r=\tau(c,w,\hat \rho)$ to obtain the bound $ \delta_\infty(c,w)\ge \tau^3(c,w,\hat \rho) \hat \rho.$

\end{proof}

\begin{lem}
\label{simple}
\begin{eqnarray}
\delta_\infty(0,2)=\delta_\infty(1,1)&=&1/2-\alpha.\label{0212}
\end{eqnarray}
\end{lem}
\begin{proof}

A centered $B$-wedge always occupies a fraction $\alpha$ of the volume of the ball $B$. The volume occupied by an additional $B$-wedge $W$ is limited by the presence of the centered $B$-wedge: since the interior of $W$ cannot contain the origin, at least one of the half-spaces defining it (let it be $H$) does not contain the center of $B$ in its interior. Therefore, $\operatorname{vol}(W)\leq \operatorname{vol}(H\cap B)\leq 1/2 \operatorname{vol}(B)$. The minimal fraction of empty space $\delta_\infty(0,2)$ in a ball $B$ occupied by a centered $B$-wedge and a $B$-wedge is therefore bounded below by $1-\alpha-1/2=1/2-\alpha$. The bound also holds for a $B$-cap (since a $B$-cap is a special case of a $B$-wedge), hence $\delta_\infty(1,1)\geq1/2-\alpha$.

Finally, this lower bound is fulfilled when the non-centered $B$-cap or $B$-wedge is a hemisphere, implying $\delta_\infty(1,1)\leq 1/2-\alpha$ and $\delta_\infty(0,2)\leq1/2-\alpha$ and the announced result.

\end{proof}

\section{Upper bound to the packing density of regular tetrahedra}

We use the recursion lemmas from Section \ref{recurs} to obtain, in Theorem \ref{capwedgebound}, a lower bound $\hat \delta_\infty(0,5)$ to the missing volume fraction $\delta_\infty(0,5)$. In Theorem  \ref{embroidery}, we express an upper bound to the packing density of tetrahedra in terms of $\delta_\infty(0,5).$ Combining these two results we obtain, in Corollary \ref{main}, an upper bound to the packing density of tetrahedra. 
\begin{thm}\label{capwedgebound}
For $c+w\leq5$ and $c\leq 2$, the missing volume  $\delta_\infty(c,w)$ is bounded below by the values shown in Table~\ref{tabbound}.
\begin{table}
\begin{center}
\begin{tabular}{|c|c|c|c|c|}
\hline
$(c,w)$&$\hat\delta_\infty(c,w)$ &$t(c,w)=\alpha w + c/2 $&$\pi(c,w)$&$(c',w')$\\
\hline\hline
$(2,1)$&$2.2\ldots \times 10^{-3}$&$1.20$&$ 0.196$&$(1,1)$\\
$(1,2)$&$5.4\ldots\times10^{-5}$&$  0.892$&$0.288$&$(2,1)$\\
$(0,3)$&$7.9\ldots\times10^{-6}$&$0.588$&$0.524$&$(1,2)$\\
$(2,2)$&$1.5\ldots\times 10^{-6}$&$1.39$&$ 0.304$&$(1,2)$\\
$(1,3)$&$4.2\ldots\times10^{-9}$&$1.09$&$0.0814$&$(0,3)$\\
$(0,4)$&$3.2\ldots\times10^{-11}$&$0.784$&$0.196$&$(1,3)$\\
$(2,3)$&$2.2\ldots\times10^{-10}$&$1.59$&$0.373$&$(1,3)$\\
$(1,4)$&$2.8\ldots\times10^{-13}$&$1.28$&$0.208$&$(0,4)$\\
$(0,5)$&$8.5\ldots \times10^{-19}$&$0.980$&$0.0144$&$(1,4)$\\
\hline
\end{tabular}
\end{center}
\caption{\label{tabbound} Approximate values of the successive bounds leading to a bound on $\delta_\infty(0,5)$ for regular tetrahedra and radius ratio of the corresponding nested spheres $\pi(c,w)$. Each bound is obtained from an earlier bound through $\hat \delta_\infty(c,w)=\pi(c,w)^3 \hat \delta_\infty(c',w').$}
\end{table}
In particular, 

\begin{equation}
\delta_\infty(0,5)\ge \hat\delta_\infty(0,5) =\left(\frac{1}{2}-\alpha\right)\prod_{(c,w)\in S}{\pi(c,w) }^3 ,
\end{equation}
where $S=\left\{ (0,5),(1,4),(0,4),(1,3),(0,3),(1,2),(2,1)\right\}$, and 

\begin{equation}
\label{pi}
\pi(c,w)\equiv\left\{\begin{array}{cc}
 \frac{2 (c/2+ w \alpha -1)}{ c+ (w-1)\sin{(\pi \alpha)}}&\textrm{if } \alpha w+\frac{c}{2}> 1\\
 \tau(c,w,\hat \rho(c,w)) &\textrm{if } \alpha w+\frac{c}{2}< 1.
\end{array} \right.
\end{equation}
Here $\tau$ is defined in Lemma \ref{lemmin}, and $\hat \rho(c,w)\equiv \min\left(\hat \delta_\infty(c-1,w),\hat \delta_\infty(c+1,w-1)\right).$

\end{thm}
\begin{proof}
The proof is obtained as in dynamic programming: values for $\delta_\infty(c,w)$ with $c+w=2$ were obtained in Lemma \ref{simple}, and lower bounds $\hat \delta_\infty(c,w)$ to $\delta_\infty(c,w)$ with $c+w>2$ are obtained by repeated application of Lemmas \ref{overfull} and \ref{lemmin}. 

Namely, if $\alpha w + c/2>1$, we define the bound:

 \begin{equation}\label{exacttobound}
\begin{split}
\hat \delta_\infty(c,w)&\equiv \kappa(c,w)^3\times \left\{\begin{array}{cc} \hat \delta_\infty(c-1,w) & \operatorname{if} c\neq 0 \\
\hat \delta_\infty(0,w-1) & \operatorname{otherwise},
\end{array}
\right. \\
&\leq
\kappa(c,w)^3\times \left\{\begin{array}{cc} \delta_\infty(c-1,w) & \operatorname{if} c\neq 0 \\
\delta_\infty(0,w-1) & \operatorname{otherwise},
\end{array}
\right.\\
&\leq
\delta_\infty(c,w),
\end{split}
\end{equation}
where the last inequality results from Lemma \ref{overfull}, Equation \eqref{ladderfull} .

Similarly, when $\alpha w + c/2<1,$  we define

\begin{equation}\label{nonincrease}
\hat  \delta_\infty(c,w)\equiv \tau^3(c,w,\hat \rho) \hat \rho  ,
\end{equation}
with $\hat\rho=\hat \rho(c,w)\equiv \min\left(\hat \delta_\infty(c-1,w),\hat\delta_\infty(c+1,w-1)\right).$ 
We then use Equation \eqref{ladderempty} from Lemma \ref{lemmin} to show that, since $0<\hat \rho \leq \rho(c,w)\equiv \min\left(\delta_\infty(c-1,w),\delta_\infty(c+1,w-1)\right),$ the desired bound holds:  $$\hat  \delta_\infty(c,w)\leq   \delta_\infty(c,w).$$

The proof is then constructed by progressively considering configurations with increasing $c+w$ and, for each value of $c+w$, progressively increasing value of $w$. This guarantees that the bounds used in Equations \eqref{exacttobound} and \eqref{nonincrease} to calculate $\hat \delta_\infty(c,w)$ have already been calculated.

\end{proof}

\begin{thm}\label{embroidery}
The packing fraction deficit $\delta_t$ of regular tetrahedra satisfies the bound
\begin{equation}
\delta_t\ge \frac{b}{1+b},
\end{equation}
where

\begin{equation}
b=(3.0\ldots\times 10^{-7})\delta_\infty(0,5).
\end{equation}

\end{thm}

\begin{proof}
Consider a packing of $N$ unit-edge tetrahedra within a box of volume $V$. We embroider each of the $e=6$ edges of the tetrahedron with a string of $k$ (spherical) pearls. These pearls are centered on the edges, have radius $r$, and are spaced by $2r$ along each edge such that the pearls at the ends of each string have centers at equal distance $d=1/2-(k-1)r$ from the vertices of the edge (see Figure \ref{embroidfig}). We are interested in determining the maximum radius $R>r$, such that a ball $B$ of radius $R$ centered within $2r$ of any of the pearl centers intersects the tetrahedron in a $B$-wedge. This maximum $R$ is achieved if balls with radius $R+2r$ centered on pearls at the ends of the string are tangent to opposite faces of the tetrahedron (see Figure \ref{embroidfig}). From the geometry of the tetrahedron, we obtain $R=\zeta\, d - 2 r$, with $\zeta=\sqrt{2/3}$.

\begin{figure}
\begin{center}
\scalebox{.7}{\includegraphics{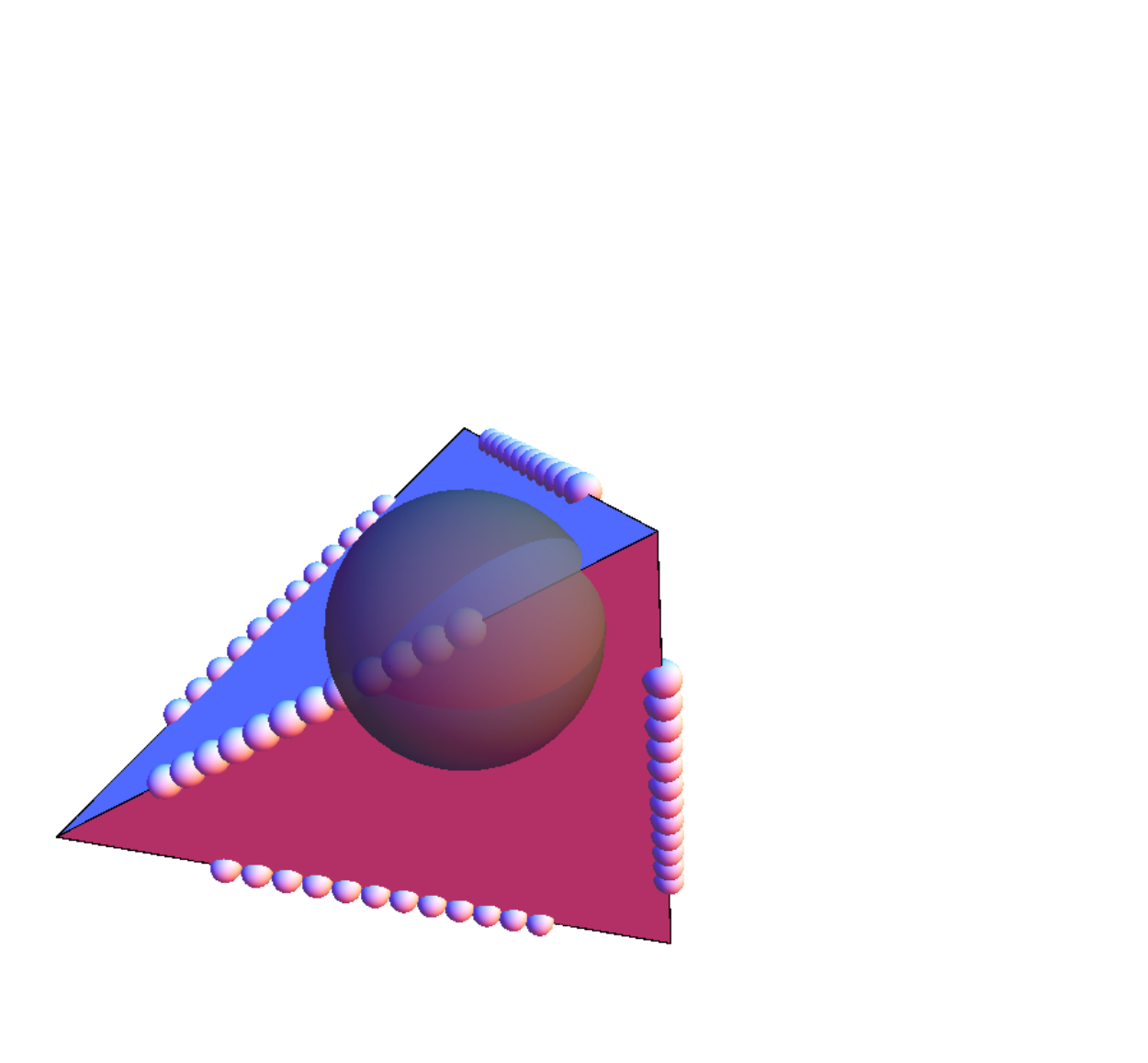}}
\end{center}
\caption{\label{embroidfig}
A unit edge tetrahedron with $k=12$ beads  of radius $r(12)\simeq0.04$ along each edge, and a ball of radius $R= 2r(12)/\gamma_7\simeq 0.3$ centered around a pearl at the end of the string. }
\end{figure}

We wish to find a lower bound to the size of the largest set of disjoint pearls in a packing of $N$ tetrahedra. Consider a particular pearl $P$ and the ball $B$ of radius $R$ concentric with it. If a pearl $P'$ on a tetrahedron $T'$ overlaps this pearl, then $T'\cap B$ is a $B$-wedge, whose edge is within a distance $2r$ from the center of $B$. Using Corollary \ref{5wedges} and imposing $(2 r)/R\leq\gamma_7\equiv 0.304$, we find that at most $s=7$ non-empty $B$-wedges can overlap with $B$, including $T\cap B$.  This limits the number of tetrahedra whose pearls overlap with this pearl to $s=7$.
Making the choice $2 r=\gamma_7 R$ (for the strongest bound) we can express $r$ in terms of $k$:

\begin{eqnarray}
d&=&\sqrt{3/2}\,(R+2 r)=\sqrt{6}\,r \left(\frac{1}{\gamma_7}+1\right)\\
&=&1/2-(k-1)r,
\end{eqnarray}
\begin{equation}
r=r(k)=\frac{1}{\sqrt{24}\left(\frac{1}{\gamma_7}+1\right)+2(k-1)}.\label{rk}
\end{equation}

Now, to show the existence of a large set of disjoint pearls, consider the graph $G$ with $ekN$ vertices corresponding to the $e kN$ pearls in the packing and edges corresponding to overlaps between pearls on different tetrahedra (pearls on the same tetrahedron do not overlap: by construction if they are on the same edge; because $d>2r$ if they are on adjacent edges; and since $r(k)<1/2 \sqrt{2}$, half the minimal distance between opposite edges). The pearl $P$ can intersect with at most two pearls on a given tetrahedron $T'\neq T$: it cannot intersect with pearls on different edges of $T'$ (since $B\cap T'$ would then not be a $B$-wedge), and it cannot intersect with non-adjacent pearls on a given edge (since these are separated by more than $2r$).  With the choice of parameters given above, each pearl can therefore intersect at most $12=2 \times (s-1)$ other pearls: the degree of $G$ has upper bound $2(s-1)$. 

We now consider the graph $G'$ obtained from $G$ by discarding all graph vertices corresponding to pearls whose interior contain a vertex. Since there are $g=4$ vertices per tetrahedra and each tetrahedron vertex can only be in one pearl per tetrahedron for at most $s$ tetrahedra, the total number of pearls removed is bounded above by $s N g$. $G'$ therefore has at least $(ek-s g)N$ vertices and degree at most  $2 (s-1)$. This guarantees the existence of an independent set of size $\lceil (ek-s g) N/(2s-1) \rceil\geq (e k-s g) N/(2s-1)$, which guarantees the existence of a set of at least $M=(e k-s g) N/(2s-1)$ disjoint pearls that are free of vertices.

Consider one of the $M$, vertex-free, disjoint pearls $P$. Any tetrahedron of the packing that overlaps $P$ will intersect in one or two of its faces (as a $P$-cap or $P$-wedge, respectively), or in more faces. The latter can be avoided if $P$ is replaced by a smaller pearl $P'$ of radius $r'=(\sqrt{2}/3) r$. Any tetrahedron intersecting $P'$ must avoid the center of $P'$ (since at least one wedge edge passes through the center) and have its nearest vertex at a distance $3/\sqrt{2}$ times the radius of $P'$. By Lemma \ref{2cw}, this limits the kinds of tetrahedron intersections with $P'$ to caps and wedges. Finally, by replacing $P'$ with yet a still smaller pearl $P''$ of radius 
$r''=\gamma_5 \, r'$, where $\gamma_5\equiv0.125,$ we use Corollary \ref{5wedges} to bound the total number of tetrahedron overlaps in $P''$ to $5$. Since $P''$-caps are a special case of $P''$-wedges, this limits $c+w\leq 5$, where $c$ is the number of non-empty $P''$-caps and $w$ the number of non-empty $P''$-wedges and, by Lemma \ref{capspecial}, the missing volume fraction of the tetrahedron packing in $P''$ is bounded by $\delta_\infty(0,5)$. The resulting volume deficit $v$ of the tetrahedron packing contributed by $P''$ has lower bound $$v_1= \delta_\infty(0,5)v_sr''^3=\delta_\infty(0,5)(\sqrt{2}\,\gamma_5/3)^3\, v_s\, r^3=a\, r^3(k),$$.



where $v_s=4\pi/3$ is the volume of the unit ball. Using our lower bound $M$ on the number of disjoint, vertex-free pearls in the packing, we bound the total volume deficit $v$ by
\begin{equation}
\begin{split}
v\ge M v_1&=a\,\frac{e k-gs}{2s-1}\, r^3(k)\,N\\
&= b\, v_t\, N,
\end{split}
\end{equation}
where $v_t=\sqrt{2}/12$ is the volume of the unit-edge tetrahedron and

\begin{equation}
b=\left((\sqrt{2}\,\gamma_5/3)^3\, \frac{ke-gs}{2s-1}r^3(k)\,\frac{v_s}{v_t}\right)\delta_\infty(0,5).
\end{equation}
We choose $k=12$ to obtain the tightest bound, which yields

$$b= \left(3.09 \ldots\times 10^{-7}\right)\delta_\infty(0,5).$$
From this we obtain a bound on the packing fraction deficit, for $N$ tetrahedra in the volume $V$:
\begin{eqnarray}
\delta_t&=&\frac{v}{V}\\
&\ge&b\,\frac{N\,v_t}{V}\\
&=&b\left(\frac{V-v}{V}\right)\\
&=&b\,(1-\delta_t),
\end{eqnarray}
and therefore $\delta_t\geq \frac{b}{1+b}.$ Since this result is independent of $N$ and $V$, we can take the limit $N, V\rightarrow \infty$ to obtain the announced result.
\end{proof}

\begin{cor}
\label{main}
The optimal packing density of regular tetrahedra $\phi$ is bounded by $\phi\leq\hat\phi=1-\delta$, with $\delta=2.6\ldots\times 10^{-25}.$
\end{cor}

\section{Application to the regular octahedron}
\label{oct}
The nested sphere approach presented here can be applied to other regular polyhedra for which other general-purpose approaches to bounding the packing density, such as the one described in \cite{Torquato:2009p747}, do not provide a nontrivial bound. 

We applied the nested sphere approach to bound the packing density of regular octahedra above by $ \hat\phi=1-\delta_o$, with $\delta_o=1.42\ldots\times 10^{-12}.$ The larger value of $\delta_o$ in the case of octahedra mostly results from a larger dihedral angle: except at vertices, the edges of at most three non-overlapping octahedra can intersect at a single point, compared with five for tetrahedral edges. The number of nested spheres to consider is therefore much reduced in the case of octahedra. 

We do not present a complete proof for the octahedron bound, but rather point out to the few differences between the arguments leading to the bounds for tetrahedron and octahedron packing. These are:
\begin{itemize}
\item The dihedral angle $2 \pi \alpha$ is replaced by $2 \pi \alpha_o$, with $\alpha_o=\arctan(\sqrt{2})/\pi\simeq 0.30$.

\item In the analogue to Lemma \ref{5wedges}, the best bound is obtained by defining $\gamma_{o3}=0.182$ and $\gamma_{o4}=0.339$ to limit the number of wedge intersections to 3 and 4, respectively. 

\item In the analogue to Lemma \ref{2cw}, the minimal distance of the center of a ball of radius $r$ to any octahedron vertex that guarantees a wedge intersection is $\eta_o r=2 r$ rather than $\eta =3 r/\sqrt{2}$ for tetrahedra.

\item Since the right-hand side of the equivalent to Equation \eqref{unique} is not a strictly non-increasing function of $r$ in the case of octahedra, a positive solution to the analogue of Equation \eqref{unique} in Lemma  \ref{lemmin} exists, but it is not necessarily unique when the octahedron dihedral angle is used. However, the non-increasing condition is satisfied for $r\leq 0\leq 1$, and there can be at most a single solution in that range. $\tau$ is then defined as this solution (if it exists), and one, otherwise. 

\item  The analogue of Theorem \ref{embroidery} depends on the geometry of the octahedron through the analogue to Lemma \ref{2cw}, the number of edges ($e_o=12$) and vertices ($g_o=6$), the total volume of a regular unit octahedron $v_o=\sqrt{2}/3$, and the minimal distance of a point on an edge to a nonadjacent face,  $\zeta_o r=\sqrt{3} r/2$, with $r$ the distance of the point to the nearest vertex.

\item We used $s_o=4$, $k_o=7$ to obtain an analogue of Theorem \ref{embroidery}, which resulted in a tighter bound for octahedra than the choice $s=7$, $k=12,$ which was optimal for tetrahedra.

\item Since the maximum number of octahedral edges that can meet at a vertex-free point is 3 rather then 5, Table \ref{tabbound} is replaced by Table \ref{taboct} and, in the analogue to Theorem \ref{embroidery}, $b$ is replaced by $b_o=(2.07\ldots\times 10^{-6})\delta_\infty(0,3).$

\end{itemize}

\begin{table}
\begin{center}
\begin{tabular}{|c|c|c|c|c|}
\hline
$(c,w)$&$\hat\delta^o_{\infty}(c,w)$ &$t_o(c,w)=\alpha_o w + c/2 $&$\pi_o(c,w)$&$(c',w')$\\
\hline\hline
$(2,1)$&$5.5\ldots \times 10^{-3}$&$1.30$&$ 0.304$&$(1,1)$\\
$(1,2)$&$3.3\ldots\times10^{-4}$&$  1.11$&$0.119$&$(0,2)$\\
$(0,3)$&$6.8\ldots\times10^{-7}$&$0.912$&$0.128$&$(1,2)$\\
\hline
\end{tabular}
\end{center}
\caption{\label{taboct} Approximate values of the successive bounds leading to a bound on $\delta^o_\infty(0,3)$ for regular octahedra, and radius ratio of the corresponding nested spheres $\pi_o(c,w)$. Each bound is obtained from an earlier bound through $\hat \delta^o_{\infty}(c,w)=\pi_o(c,w)^3 \hat \delta^o_{\infty}(c',w')$. The definitions for $\pi_o$ and $\delta^o_\infty$ in this Table are direct analogues of the definitions for tetrahedra, with octahedral wedges replacing tetrahedral wedges, i.e.,  $\alpha_o$ replacing $\alpha.$}
\end{table}

Definitions \ref{capswedges}, \ref{defd}, and \ref{order},  Lemmas \ref{wedgemax}, \ref{wedgesurf}, \ref{recurs1}, and \ref{simple}, and Corollary \ref {sm1} have direct analogues for the octahedron packing problem, once the appropriate dihedral angle has been substituted.

Given a regular or quasi-regular polyhedron, a bound can be calculated from the dihedral angle $\alpha$, the number of edges $e$ and vertices $g$, the volume of the unit-edge polyhedron $v$, and the geometry parameters $\eta$ and $\zeta$. All other quantities (such as $k$ and $s$) are derived from these $6$ values.

\section{Conclusion and possible improvements to the bound}
We presented an elementary proof of an upper bound to the packing density of regular tetrahedra and octahedra. These bounds are not tight; for the sake of simplicity, we have made many choices that resulted in a sub-optimal bounds. Straightforward improvements would result from using the exact bounds for $B$-wedge and $B$-cap volumes rather than the simpler bounds we used  (see Figure \ref{penguin}). More significant improvements are likely to come from more profound modifications to the argument.

The argument leading to the bound presented here is essentially local, in that it considers a particular set of well-separated points and independently bounds the missing volume in small neighborhoods around each point, without consideration for violation of the packing condition away from the neighborhood. The proof can therefore easily be transposed to any polyhedron that cannot subtend a solid angle of $4\pi$ around edges (the generalization to vertices is also straightforward). Even though it is likely that the bounds can be improved by many orders of magnitude through such local arguments (for example, by finding the exact value for  $\delta_\infty(0,5)$, which we conjecture to be $1-5\alpha =0.0204\ldots$), it is likely that consideration for nonlocal effects will be crucial in obtaining a tight bound. Such nonlocal effects can be taken into account by considering the effect of packing conditions outside the well-separated neighborhoods (in effect, considering longer range interactions between tetrahedra) or through the use of larger neighborhoods.

Finally, the use of numerical exploration and enumeration was crucial in the identification of dense tetrahedron packings and in the proof of Kepler's conjecture, and we suggest that a numerical method to obtain stronger upper bounds, as an intermediate step en route to obtaining a tight upper bound, would be a useful tool to understand the packing behavior of tetrahedra and, more generally, of granular matter.

\appendix
\section{Proof of the geometric lemmas}

\begin{proof}[Proof of Lemma \ref{wedgemax}] The result for $B$-caps is straightforward. For $B$-wedges, we first show that the volume is maximized when one face of the $B$-wedge contains the origin, then bound the volume of such a wedge by breaking it down in three subsets as illustrated on Figure \ref{3p1}.

Since the $B$-wedge $W$ is the intersection of two half-spaces, $H_1$ and $H_2$, with $B$, one of the half-spaces (let it be $H_1$) must have the center of $B$ outside its interior. Consider the $B$-wedge $W'$ given by the intersection of $B$ with the half-spaces $H_1'=t(H_1)$ and $H_2'=t(H_2)$, where t is a translation along the direction perpendicular to $H_1$ by a distance such that the boundary of $H_1'$ contains the center of $B$. Since $\alpha\leq1/2$, the interior of $W'$ does not contain the center of $B$. We have  $t(W)\subseteq W'$, which implies $\operatorname{vol}(W')\ge\operatorname{vol}(W)$, and the edge of $W'$ is a distance $d\le r$ from the center of $B$.

If the surface of $W'$ does not contain the center of $B$, one can define a centered $B$-wedge $W^0\supset W'$ by translating $H_2',$ so that $\operatorname{vol}(W')\leq4 \pi \alpha/3$ and the lemma holds (see Figure \ref{3p1}).
\begin{figure}
\scalebox{.5}{\includegraphics{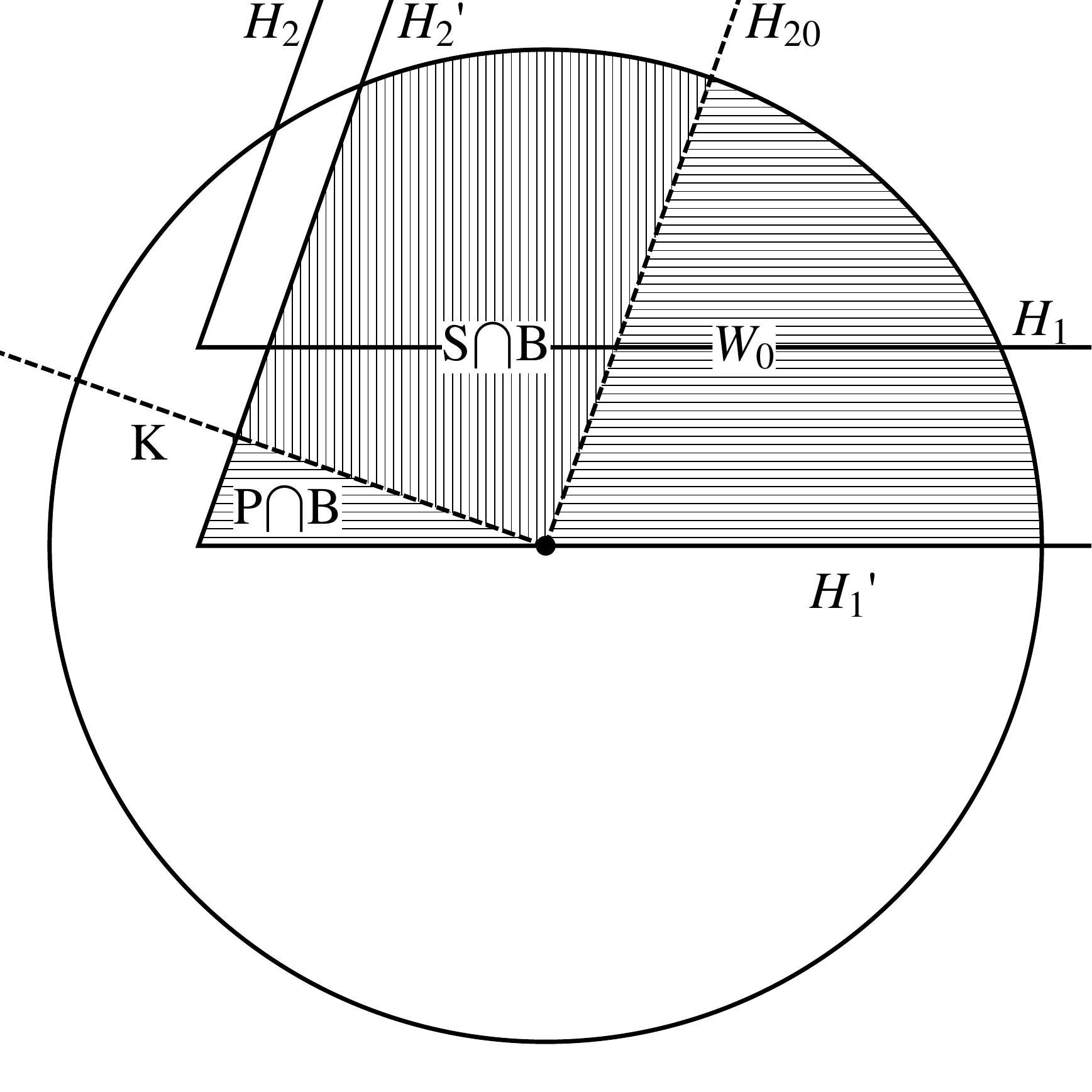}}
\caption{\label{3p1} A cross-section of the ball $B$, containing the center of B and orthogonal to the edge of the $B$-wedge $W$, together with different constructions used in the proof of Lemma \ref{wedgemax}. Depicted are the initial $B$-wedge $W=H_1\cap H_2$, with faces labeled by the name of the corresponding half-spaces, and  $W'=H_1'\cap H_2'$, shaded to illustrate the three subsets ($P\cap B,$ $ S\cap B $, $W_0$) separated by plane $K$ and the boundary of the half-space $H_{20}$.  }
\end{figure}
 Otherwise, we separate $W'$ into two components; a centered $B$-wedge $W^0$ (defined by a half-space $H_2^0$ parallel to $H_2$ and the half-space $H_1'$),  and a volume $J,$ delimited by the surfaces of $H_2'$, $H_2^0$, $H_1'$, and $B$. This volume can be further divided by a plane $K$ going through the edge of $W_0,$ and orthogonal to the surface of $H_2.$ This allows us to write $J=(S\cap B) \cup (P\cap B),$  where $S$ is half of the cylinder of unit radius bounded by the parallel boundaries of $H_2'$ and $H_2^0$, $P$ is a prism with base area $d^2 \sin(4\pi \alpha)/4$ and height 2, and $S\cap (P\cap B)=\emptyset$. We therefore have      
\begin{equation}
\begin{split}
\operatorname{vol}(W')&=\operatorname{vol}(W^0)+\operatorname{vol}(S\cap B)+\operatorname{vol}(P\cap B)\\& \leq \operatorname{vol}(W^0)+\operatorname{vol}(S)+\operatorname{vol}(P)\\
&=\frac{4 \pi \alpha}{3}+\frac{\pi d}{2}  \sin(2 \pi \alpha) +\frac{d^2}{2} \sin(4 \pi \alpha).
\end{split}
\end{equation} 
For the tetrahedron dihedral angle, this reduces to 

$$
\frac{\operatorname{vol}(W')}{\operatorname{vol}(B) } \le \alpha+\frac{3 d}{8} \sin(2 \pi \alpha) +\frac{3 d^2}{8 \pi} \sin(4 \pi \alpha)=\alpha+\sqrt{2}\left(\frac{d}{4}+\frac{d^2}{6 \pi} \right) ,
$$
and the bound is loosest when $d=r,$ yielding the stated result.



\end{proof}

\begin{proof}[Proof of Lemma \ref{wedgesurf}]

The bound for $B$-caps follows directly from Archimedes' hat-box theorem: $ \frac{1}{2}\left(1-r\right)$ is the fractional  area of the spherical zone $C\cap\partial B$, when $C$ is at a distance $r$ from the center of $B$.

For $B$-wedges, we first show that the area $\sigma(W\cap\partial B)$ reaches a minimum for $W=W^\star$, where $W^\star$ is a $B$-wedge whose bisecting plane contains the origin, whose interior does not contain the origin and whose edge is at a distance $r$ from the origin.

We first observe that since $W$ intersects $B_r$, we can define a $B$-wedge $W' \subseteq W$ whose edge $E$ is at a distance $r$ from the center of $B$. Using the symmetries of the ball $B$, any wedge whose edge is at a distance $r$ from the origin can be mapped to a wedge $W_E,$ with edge $E$, with no modification to the area of intersection with $\partial B$.  We therefore have  $\sigma (W\cap \partial B)\geq \sigma (W'\cap \partial B) = \sigma(W_E \cap \partial B)$.

We therefore focus our attention on $B$-wedges with edge $E$. The area of intersection of one such wedge $W_E$ with $\partial B$ can be calculated by integrating over a coordinate $z$ along the edge:

$$\sigma(C\cap\partial B)= \int_{-1}^1 \frac{\ell(z)}{\sqrt{1-z^2}} dz,$$
where $\ell(z)$ is the length of the arc defined by $W_E\cap \partial B\cap R_z$, with $R_z$ a plane orthogonal to $E,$ that is, with constant $z.$ We then consider the two-dimensional problem of minimizing the arc length along a circle $\partial B \cap R_z$, defined by an angle with fixed opening angle, $2 \pi \alpha$, and a fixed vertex, $E\cap R_z$. Showing that this is achieved by the angle whose interior does not contain the center of $\partial B \cap R_z$ and whose bisector contains the center of $\partial B \cap R_z$ is an exercise in planar geometry. $W^\star$ therefore minimizes the integrand for each value of $z$, and we have $\sigma(W\cap \partial B)\geq \sigma(W'\cap \partial B)\geq \sigma(W^\star\cap \partial B).$





The area of the spherical figure $W^\star \cap\partial B$ is related to the curvature of its boundary by the Gauss-Bonnet theorem. This boundary comprises two circular arcs, each with net curvature $\beta$, joined at the vertices $E\cap\partial B$ where the curve has internal angle $\theta$. By the Gauss-Bonnet theorem, $A=\sigma(W^\star\cap \partial B)=2\theta-2\beta$. 
Elementary trigonometry gives
\begin{eqnarray}
\beta=\beta(r)&=&2r\sin{(\pi\alpha)}\arctan{\left(\frac{\sqrt{1-r^2}}{r\cos{(\pi\alpha)}}\right)}\\
\theta=\theta(r)&=&2\arctan{\left(\tan{(\pi\alpha)}\sqrt{1-r^2}\right)}.
\end{eqnarray}
Defining the fractional area $a(r)=\frac{\sigma(W^\star \cap\partial B)}{\sigma(\partial B)}= (2\theta(r)-2\beta(r))/(4\pi),$ we find

$$a''(r)=\frac{\text{sin}(2 \pi  \alpha )}{\pi  \sqrt{1-r^2} \left(2-r^2 \sin(2 \pi  \alpha)\right)}.$$ 

Since $\alpha \leq 1/2,$ this is monotone increasing for $0<r<1$,  and we can replace $a(r)$ by the bound
\begin{eqnarray}
a(r)&\geq&a(0)+r a'(0)+r^2 a''(0)/2\\
&=& \alpha -\frac{r \sin(\pi \alpha)}{2} + \frac{r^2}{4\pi} \sin(2 \pi \alpha)\\
&=&\alpha-\frac{r}{2 \sqrt{3} } + \frac{\sqrt{2} r^2}{6 \pi}.
\end{eqnarray}

\end{proof}

\begin{proof}[Proof of Lemma \ref{2cw}]

$T\cap B_r$ is a $B_r$-wedge if and only if the interior of $B_r$ intersects at most two faces of $T$.

Consider the three faces of $T$, $F_i,~~i=1,2,3,$ intersecting at vertex $V$, and $Y$ the infinite intersection of the corresponding three half-spaces. We also define the faces of $Y$,  $\tilde F_i,~~i=1,2,3$, which extend the $F_i$ away from $V$.  We first find the point $P^\star$ outside $Y$ and at a fixed distance $d$ from $V$ that minimizes the distance to the farthest of the three $\tilde F_i$. Without loss of generality, let us suppose that face $\tilde F_1$ is closest to $P^\star$. By symmetry, $P^\star$ must lie on the bisector $D$ of $\tilde F_2$ and $\tilde F_3$. Otherwise, a translation towards $D$ would reduce the distance to the farthest face (such a translation is possible since $\tilde F_1\perp D$ and $\tilde F_1$ is the closest face to $P^\star$). Parametrize $D$ by coordinates $(x,y)$ such that the origin is at $V$ and $(x,0)$ parametrizes the edge $ F_2\cap  F_3$ for $-1\leq x\leq 0$.

We want to show that $P^\star=P\equiv \left(-d/\sqrt{3},-\sqrt{2/3}d \right),$ that is, a point on $D\cap \tilde F_1$ that lies at a distance $d(P, \tilde F_{2,3})=\sqrt{2}d /3$ from $\tilde F_2$ and $\tilde F_3$. To see that this point achieves the minimal distance, consider an alternate point $P' =(p'_x,p'_y)\in D$, and the closest point $Q'$ to $P'$ on $\tilde F_2.$ If $p'_y\geq 0, p'_x\geq 0,$ then $Q'$ is the origin and the distance is $d(P',\tilde F_{2,3})=d>\sqrt{2}d /3$. If $p'_y\geq 0, p'_x<0,$, the face $\tilde F_1$ is not the closest of the three faces. Finally, if $p'_y<0,$ the vector $A'=Q'-P'$ can be decomposed into orthogonal components parallel and perpendicular to $D:$ $A'= A'_\parallel+A'_\perp$. Similarly, the vector $A=Q'-P$ can be decomposed in components parallel and orthogonal to $D:$ $A= A_\parallel+A_\perp$. We then have $$A_\perp =A'_\perp,$$ $$E= P+A_\parallel=P'+A'_\parallel,$$  and, since $|A'_\parallel|^2  =|E|^2 + d^2- 2 P'\cdot E \geq|E|^2 + d^2- 2 P\cdot E = |A_\parallel|^2$,  $$|A |\leq |A'|.$$
 Therefore  $\sqrt{2}d /3=d(P,\tilde F_2)\leq d(P,Q')\leq d(P',Q'),$ and $P^\star=P$.  

Now suppose that the interior of ball $B_r$ intersects all three $F_i.$ It then also intersects all three $\tilde F_i$, and if the center of $B$ is outside $Y$, the argument above guarantees that the distance between the center of $B_r$ and $V$ is strictly less than $3 r/\sqrt{2}$. If the center of $B_r$ is in $Y\setminus T$, the center of $B_r$ is closest to $F_4$, the fourth face of $T$. In this case $B_r$ intersects all four faces of $T$. Since the center of $B_r$ is in $Y\setminus T$, it is outside the half space defining $F_4$; the argument above can therefore be applied to any vertex $V'\neq V$; the other three vertices are at a distance less than $3 r/\sqrt{2}$ from the center of $B_r.$ Therefore the condition that the center of $B$ lies at a distance of at least  $3 r/\sqrt{2}$ from any vertex of $T$ ensures that the interior of $B$ intersects at most two faces, and $T\cap B_r$ is a $B_r$-wedge.

\end{proof}

\bibliography{tetrarefs}{}

\begin{thebibliography}{10}

\bibitem{Hales:2005p874}
TC~Hales.
\newblock {A proof of the Kepler conjecture}.
\newblock {\em Annals of Mathematics-Second Series}, 162:1065, 2005.

\bibitem{Hales:2010p949}
TC~Hales, J~Harrison, S~McLaughlin, T~Nipkow, S~Obua, and R~Zumkeller.
\newblock A revision of the proof of the {K}epler conjecture.
\newblock {\em Discrete and Computational Geometry}, 44:1, 2010.

\bibitem{Euclid:1908p976}
Euclid, TL~Heath, and JL~Heiberg.
\newblock {\em The thirteen books of Euclid's Elements: Books X-XIII and
  appendix‎}.
\newblock Cambridge University Press, 1908.

\bibitem{Senechal:1981p312}
M~Senechal.
\newblock Which tetrahedra fill space?
\newblock {\em Mathematics Magazine}, 54(5):227, 1981.

\bibitem{DeCaelo}
Aristotle.
\newblock {\em On the Heavens}, volume III.
\newblock ebooks@adelaide.edu.
\newblock Translated by J. L. Stocks.

\bibitem{Hilbert:1900p1142}
DC~Hilbert.
\newblock Mathematische probleme.
\newblock {\em Nachr. Ges. Wiss. G{\"o}tt., Math. Phys. Kl.}, 3:253, 1900.

\bibitem{Conway:2006p571}
JH~Conway and S~Torquato.
\newblock Packing, tiling, and covering with tetrahedra.
\newblock {\em Proc. Natl Acad. Sci. USA}, 103(28):10612, 2006.

\bibitem{HajiAkbari:2009p610}
A~Haji-Akbari, M~Engel, A~Keys, and X~Zheng.
\newblock Disordered, quasicrystalline and crystalline phases of densely packed
  tetrahedra.
\newblock {\em Nature}, 462:773, 2009.

\bibitem{Torquato:2009p747}
S~Torquato and Y~Jiao.
\newblock Dense packings of the platonic and archimedean solids.
\newblock {\em Nature}, 460:876, 2009.

\bibitem{Torquato:2009p714}
S~Torquato and Y~Jiao.
\newblock Dense packings of polyhedra: Platonic and archimedean solids.
\newblock {\em Phys. Rev. E}, 80:041104, 2009.

\bibitem{Kallus:2010p497}
Y~Kallus, V~Elser, and S~Gravel.
\newblock Dense periodic packings of tetrahedra with small repeating units.
\newblock {\em Discrete and Computational Geometry}, 44:245, 2010.

\bibitem{kallus2010method}
Y~Kallus, V~Elser, and S~Gravel.
\newblock {A method for dense packing discovery}.
\newblock {\em Arxiv preprint arXiv:1003.3301}, 2010.

\bibitem{Chen:2008p521}
ER~Chen.
\newblock A dense packing of regular tetrahedra.
\newblock {\em Discrete Comput. Geom.}, 40:214, 2008.

\bibitem{PhysRevE.81.041310}
S~Torquato and Y~Jiao.
\newblock Exact constructions of a family of dense periodic packings of
  tetrahedra.
\newblock {\em Phys. Rev. E}, 81:041310, 2010.

\bibitem{Chen:2010p783}
ER~Chen, M~Engel, and S~Glotzer.
\newblock Dense crystalline dimer packings of regular tetrahedra.
\newblock {\em Discrete Comput. Geom.}, 44:253, 2010.

\bibitem{Hoylman:1970p1203}
DJ~Hoylman.
\newblock The densest lattice packing of tetrahedra.
\newblock {\em Bull. Amer. Math. Soc.}, 76:135, 1970.

\bibitem{Jaoshvili:2010p1551}
A~Jaoshvili, A~Esakia, M~Porrati, and PM~Chaikin.
\newblock Experiments on the random packing of tetrahedral dice.
\newblock {\em Phys. Rev. Lett.}, 104:185501, 2010.

\end{thebibliography}
\bibliographystyle{unsrt}

\end{document}